\documentclass[11pt,reqno]{amsart}
\usepackage{graphicx} 
\usepackage{amsmath}
\usepackage{amsfonts}
\usepackage{amssymb}
\usepackage{xcolor}
\usepackage{soul}
\usepackage{amsthm}
\usepackage{hyperref}
\usepackage{rotating}
\usepackage{enumitem}
\usepackage[a4paper,margin=1.25in,headheight=14pt]{geometry}

\def\R{\mathbb{R}}

\def\st{\, | \,}
\def\dt{\dotsc}

\def\fnm{\mathcal F_n^{\raisebox{1pt}{$\scriptstyle -$}}}

\DeclareMathOperator{\Ind}{Ind}
\DeclareMathOperator{\Res}{Res}

\theoremstyle{plain}
\newtheorem{thm}{Theorem}
\newtheorem{lemma}[thm]{Lemma}

\newtheorem{defi}[thm]{Definition}
\newtheorem{innercustomthm}{Theorem}
\newtheorem{question}[]{Question}

\newenvironment{customthm}[1]
{\renewcommand\theinnercustomthm{#1}\innercustomthm}
{\endinnercustomthm}
  
\theoremstyle{remark}

\newtheorem*{conj*}{Conjecture}
\newtheorem*{thm*}{Theorem}

\usepackage{parskip}
\begin{document}
\title{Spectral gap for the signed interchange process with arbitrary sets}
\author{Gil Alon}
\email{gilal@openu.ac.il}
\address{Department of Mathematics and Computer Science, The Open University of Israel}

\author{Subhajit Ghosh}
\email{gsubhajit@alum.iisc.ac.in}
\address{Department of Mathematics, Indian Institute of Technology Madras, Chennai 600 036, India}
\maketitle
\begin{abstract}
In 2020, F.Cesi introduced a random walk on the hyperoctahedral group $B_n$ and analysed its spectral gap when the allowed generators are transpositions and diagonal elements corresponding to singletons. In this paper we extend the allowed generators to transpositions and any diagonal elements, and characterise completely the set of representations from which the spectral gap arises. This settles a conjecture posed in Cesi's paper.
\end{abstract}
\section*{Introduction}

Let $G$ be a finite group, $S$ a symmetric set of generators ($S=S^{-1}$), and $(w_s)_{s\in S}$ a symmetric set of nonnegative weights ($w_s=w_{s^{-1}}$ for all $s \in S$). Let us attach a Poisson clock with rate $w_s$ to each generator $s\in S$, and consider the following continuous time random walk on the Cayley graph $\text{Cay}(G,S)$: We start, at time $t=0$, at the unit element $1_G$. Whenever the clock of a generator $s$ rings, we multiply the existing group element by $s$ from the left.

It has been long understood that the spectral gap of this process can be analysed by using the representation theory of $G$. Indeed, let $w=\sum w_s s$ (an element of the group ring $\R[G]$), and define its Laplacian by $\Delta(w)=\sum_{s\in S} w_s(1_G-s)$. For any finite dimensional complex representation $\sigma$ of $G$, the matrix $\sigma(
\Delta(w))$ is self-adjoint and nonnegative, and thus it has real, nonnegative eigenvalues. Such an eigenvalue is called trivial if the corresponding eigenvector is invariant under all the elements of $G$. Following Cesi (\cite{Cesi-octopus} and \cite{Cesi}), we will denote by $\psi_G(w,\sigma)$ the minimum nontrivial eigenvalue of $\sigma(\Delta(w))$.

Let us also denote
\begin{equation}
    \label{eq:def_psi_G}
    \psi_G(w)=\min_{\sigma} \psi_G(w,\sigma)
\end{equation}
where the minimum is taken over all nontrivial irreps (= irreducible representations) of $G$. Then, it is well known that the spectral gap of the continuous time random walk described above is $\psi_G(w)$.

It is natural to ask, given $G$ and $S$, whether all the nontrivial irreps are required in the minimum \eqref{eq:def_psi_G}, or it is enough to take the minimum on a subset of the nontrivial irreps. From a computational point of view, finding a subset (as small as possible) from which the spectral gap arises can save a lot of time in the calculation of the spectral gap. 

Let us specialize to the case where $G$ is the symmetric group $S_n$, and $S$ is the set of transpositions $\{(ij) \st 1\leq i < j \leq n\}$. In this case the random walk is called \textit{the interchange process}. It was first analysed (in the mean-field case) by Diaconis and Shahshahani (\cite{DS}).

Recall that the irreps of $S_n$ are parametrized by partitions of $n$, or, equivalently, Young diagrams of size $n$. In 1992, David Aldous famously conjectured (\cite{Ald._conj}) that for any choice of weights for the transpositions, the spectral gap always arises from a single irrep -- the one corresponding to the partition $[n-1,1]$.

After many attempts and partial results by various authors, the Aldous conjecture was proved in 2009 by Caputo, Liggett and Richthammer (\cite{CLR}). This positive result warrants the question, whether similar results can be obtained for other groups and sets of generators. More generally, one can ask whether for specific families of group ring elements $w=\sum_{g\in G} w_g g$, the spectral gap $\psi_G(w)$ arises from a restricted set of irreps.

These questions have been in the focus of several works in the past 16 years. Caputo has conjectured that for $G=S_n$, if $w$ is a nonnegative combination of elements of the form $\alpha_A=\sum_{\pi \in S_A} \pi$ (where $A$ is any subset of $\{1,\dt,n\}$), then the spectral gap $\psi_{S_n}(w)$ comes from $[n-1,1]$. Note that this is a generalization of the Aldous conjecture. Some partial cases of this conjecture have been proved by Bristiel and Caputo (\cite{BC24}) and by Kozma, Puder and the first author (\cite{AKP}), but the general case remains open. Parzanchevski and Puder proved (\cite{PP}) that when $w$ is the sum of the elements in a single conjugacy class of $S_n$, then the spectral gap $\psi_{S_n}(w)$ comes from one of 8 representations. 

In 2020, Cesi (\cite{Cesi}) has proved a similar result for the hyperoctahedral group $G=B_n$, which we shall now describe in detail.
Recall that $B_n$ is defined as the wreath product $G = C_2 \wr S_n$. Conveniently, $G$ can be considered as the group of $n\times n$ integer valued matrices, where each row and each column has only one nonzero entry, which is equal to either $1$ or $-1$. 
This group contains, in particular, the subgroup $S_n$, realized as the permutation matrices, and the normal subgroup $N_n \cong (C_2)^n$, realized as the diagonal matrices with $\pm 1$ entries. For each subset $A\subseteq \{1,\dt,n\}$, we shall denote by $s_A$ the element of $N_n$ whose diagonal entries are $-1$ exactly at the locations determined by $A$.

The isomorphism classes of irreps of $G$ are parametrized by pairs of Young diagrams $(\sigma_1,\sigma_2)$ where $|\sigma_1|+|\sigma_2|=n$. When no risk of confusion arises, we shall sometimes write $(\sigma_1, \sigma_2)$ to denote an arbitrary irrep corresponding to it. We note that $([n],\emptyset)$ corresponds to the trivial irrep.

Cesi's main result is the following:

\begin{thm*}[F. Cesi, 2020]
    Let $w=w_T+w_N$ be an element of the group ring $\R[B_n]$, where $w_T$ is a nonnegative combination of transpositions $(ij)\in S_n$, and $w_N$ is a nonnegative combination of the elements $s_{\{i\}}$, for $1\leq i \leq n$. Then we have
    $$ \psi_{B_n}(w)=\psi_{B_n}(w, P_n).$$
\end{thm*}

Here, $P_n$ is a $2n$--dimensional representation of $B_n$ whose irreducible components correspond to the pairs $([n-1,1],\emptyset)$, $([n-1],[1])$ and $([n],\emptyset)$.

This result was later generalized by the second author (\cite{Ghosh}) to any wreath product of the form $G \wr S_n$, where $G$ is a finite group.

At the end of Cesi's paper, he asked whether this result could be extended to the case where the generators include the transpositions and all diagonal elements. He showed that this does not hold in general (with the representation $P_n$). He made the following, more restricted conjecture:

\begin{conj*}[F. Cesi, 2020]\label{conj:Cesi}
    Let $w=w_N^-+w_T$ be an element of the group ring $\R[B_n]$, where $w_T$ is a nonnegative combination of transpositions $(ij)\in S_n$, and $w_N^-$ is a nonnegative combination of the elements $s_A$, for the subsets $A \subseteq [1,\dt,n]$ of odd size. Then we have
    $$ \psi_{B_n}(w)=\psi_{B_n}(w, P_n).$$
\end{conj*}

In the current paper, we treat both Cesi's conjecture and the general problem of finding the spectral gap $\psi_{B_n}(w)$ for any nonnegative combination $w$ of transpositions and diagonal elements.

Let us consider the following collections:
$$ \mathcal F_n = \{ ([n-1,1],\emptyset)\} \cup \{ ([n-k],[k]) \st 1 \leq k \leq n\}$$
and

$$ \fnm = \{ ([n-1,1],\emptyset)\} \cup \{ ([n-k],[k]) \st 1 \leq k < n\} = \mathcal F_n \setminus \{(\emptyset, [n])\}$$

We will prove:
\begin{thm} \label{thm:main}
    Let $w_T=\sum_{1\leq i<j\leq n} a_{ij}(ij)$ and $w_N=\sum_{A \subseteq \{1,\dt,n\}}\alpha_As_A$ with any nonnegative weights $(a_{ij})$ and $(\alpha_A)$. Let $w_N^-=\sum_{A \subseteq \{1, \dt, n \} \, : \, |A| \text { is odd} } \alpha_A s_A$.
    
    Then we have:
    \begin{enumerate} 
    \item $ \psi_{B_n}(w_T+w_N)=\min \{ \psi_{B_n}(w_T+w_N, \sigma) \st \sigma \in \mathcal F_n\}.$
    \item $ \psi_{B_n}(w_T + w_N^-)=\min \{ \psi_{B_n}(w_T+w_N^-, \sigma) \st \sigma \in \fnm\}.$
    \end{enumerate}
    \end{thm}
As we shall see in Section \ref{sec:reps}, each pair $([n-k],[k])$ corresponds to an $\binom nk$ dimensional representation which has a natural basis corresponding to subsets of $\{1,\dt, n\}$ of size $k$. In particular, the total dimension of the representations corresponding to $\mathcal F_n$ is $2^n+n-2$. We therefore get an analogy between our theorem and the Aldous-Caputo-Liggett-Richthammer theorem: The latter reduces the search for the spectral gap from the regular representation (of dimension $n!$) to a representation of dimension $n-1$, whereas our theorem reduces the search from $2^n \cdot n!$ (= the size of $B_n$) to roughly $2^n$.

Our proof of Theorem \ref{thm:main} relies on the following result, which may be of independent interest:

\begin{thm} \label{thm:stronger_main}
    Let $\sigma$ be a nontrivial irrep of $B_n$ which does \textbf{not} correspond to any element of $\mathcal F_n$, and let $w=w_T+w_N$, where $w_T$ and $w_N$ be as in Theorem \ref{thm:main}. Then we have
    $$ \psi_{B_n}(w,\sigma) \geq \psi_{B_n}(w, ([n-1,1],\emptyset)). $$
\end{thm}

The proof of Theorem \ref{thm:stronger_main} is based on the branching properties of the pair $(B_n,S_n)$ and on the proven Aldous conjecture. Interestingly, it does not require the use of a new octopus inequality: The only use of such operator inequalities is done indirectly through the use of the Aldous-Caputo-Ligget-Richthammer Theorem.

Our next result shows that all of the irreps in $\mathcal F_n$ (and respectively, $\fnm$) are required in Theorem \ref{thm:main}:

\begin{thm} \label{thm:required_irreps}
\begin{enumerate}
\item
    Let $\sigma \in \mathcal F_n$. Then there exist $w_T, w_N$ as in Theorem \ref{thm:main} such that $\tau=\sigma$ is the unique minimizer of the expression $\psi_{B_n}(w_T+w_N,\tau)$ among all nontrivial irreps $\tau$ of $B_n$.
\item 
    Let $\sigma \in \fnm$. Then there exist $w_T, w_N^-$ as in Theorem \ref{thm:main} such that $\tau=\sigma$ is the unique minimizer of the expression $\psi_{B_n}(w_T+w_N^-,\tau)$ among all nontrivial irreps $\tau$ of $B_n$.    
    \end{enumerate}
\end{thm}

In particular, Theorem \ref{thm:required_irreps} shows that Cesi's conjecture is false.

The main part in the proof of Theorem \ref{thm:required_irreps} is the case where $\sigma=([n-k],[k])$ for some $1\leq k \leq n-1$. In that case we use Fourier theory on the hypercube to construct $w_N^-$ such that $\psi_G(w_N^-,\sigma) < \psi_G(w_N^-,([n-i],[i]))$ for all $0< i\neq k$, and then show that $w_T$ can be chosen as a certain multiple of the complete graph element $\sum_{1\leq i<j\leq n}(ij)$ to satisfy the required property. 

The rest of the paper is organized as follows: In Section \ref{sec:reps} we recall and prove some properties of irreps of $B_n$. In particular, we give a concrete realization of the irreps $([n-i],[i])$. In Section \ref{sec:proof of main} we analyse the restriction from $B_n$ to $S_n$, and prove Theorems \ref{thm:main} and \ref{thm:stronger_main}. In Section \ref{sec:required irreps} we analyse the restriction from $B_n$ to $N_n$, and prove Theorem \ref{thm:required_irreps}. In the final section, we propose some open questions.
 
\section{Representations of $B_n$} \label{sec:reps}
As in the introduction, we denote by $B_n$ the hyperoctahedral group, realized as $n\times n$ signed permutation matrices, and let $N_n$ be the subgroup of diagonal matrices. Hence, $N_n \cong (C_2)^n$ and $B_n=N_n \rtimes S_n$.
Let us recall some of the representation theory of this group. 

We start by defining a multiplicative character $\nu_n$ of $B_n$: It is given on the $n\times n$ matrices by
\begin{equation} 
\label{eq:nu_k_rep}
\nu_n(A)=\prod_{1\leq i,j \leq n \text{ : } A_{ij}\neq 0}A_{ij}.
\end{equation}

For each partition of $n$, $\sigma \vdash n$, let $V_\sigma$ be a corresponding $S_n$--irrep.

The irreps of $B_n$ correspond to pairs of partitions $(\sigma_1,\sigma_2)$, where $|\sigma_1|+|\sigma_2|=n$. The correspondence can be described as follows:

Let $i=|\sigma_1|$, so $n-i=|\sigma_2|$. Pulling back by the quotient map $B_i \rightarrow B_i/N_i \cong S_i$, we consider $V_{\sigma_1}$ as a representation of $B_i$. Similarly, we consider $V_{\sigma_2}$ as a representation of $B_{n-i}$. Then (Cf. \cite{Cesi}, (3.1), and \cite{Serre}, Section 8.2), the $B_n$--irrep corresponding to the pair $(\sigma_1,\sigma_2)$ is
\begin{equation}
\label{eq:Bn_irrep_def}    
 V_{(\sigma_1,\sigma_2)}:=\Ind^{B_n}_{B_i \times B_{n-i}} ((V_{\sigma_1} \boxtimes (V_{\sigma_2} \otimes \nu_{n-i})))
\end{equation}

The following Lemma follows directly from the definitions:

\begin{lemma} \label{lem:(sigma,empty)}
    Let $\sigma \vdash n$, and consider the representation $V=V_{(\sigma,\emptyset)}$. Then $V$ is equal to the pullback of the $S_n$--irrep $V_\sigma$ via the quotient map from $B_n$ to $S_n$. In particular, we have
    $$ \Res^{B_n}_{S_n}V \cong V_\sigma.$$
\end{lemma}

The following family of $B_n$--irreps is central in our study:
    
\begin{defi}\label{def:Vni}
    For any $0\leq i \leq n$, let $V_n^i$ be the following representation of $B_n$: As a vector space, $V_n^i$ is the space of formal linear combinations of sets of size $i$ inside $\{1,\dt,n\}$,
    $$ V_n^i = \{\sum_{A \subseteq \{1,\dt,n\},\, |A|=i} \alpha_A A \st \forall A, \alpha_A \in \R\}$$ 
    The action of $B_n$ is defined by:
    \begin{enumerate}
        \item  For any $\pi \in S_n$ and $A\subseteq \{1,\dt,n\}$ of size $i$, $\pi\cdot A= \pi(A)=\{\pi(j) \st j \in A\}$.
        \item For any $A,B \subseteq \{1,\dt,n\}$, 
        \begin{equation}\label{eq:sA_action}
        s_A \cdot B = (-1)^{|A\cap B|} B.
        \end{equation}
    \end{enumerate}
    \end{defi}

\begin{lemma}\label{lem:[n-i],[i]}
    For any $0\leq i \leq n$, the action of $B_n$ on $V_n^i$ in Definition \ref{def:Vni} is well defined. It makes $V_n^i$ an irrep of $B_n$, corresponding to the pair $([n-i],[i])$.
\end{lemma}
\begin{proof}
    Consider the trivial representations $V_{[n-i]}$, $V_{[i]}$ of $S_{n-i}$ and $S_i$, respectively. Pulling back through the quotient maps $B_{n-i}\rightarrow S_{n-i}$ and $B_i\rightarrow S_i$, we view them as the trivial representations of $B_{n-i}$ and $B_i$, respectively. 
    As $V_{[n-i]}$, $V_{[i]}$ and $\nu_i$ are one dimensional, so is the $B_{n-i}\times B_i$--irrep $V_{[n-i]}\boxtimes (V_{[i]}\otimes \nu_i)$, and the induction 
    $$V_{([n-i],[i])}=\Ind_{B_{n-i}\times B_i}^{B_n} (V_{[n-i]}\boxtimes (V_{[i]} \otimes \nu_i)) \cong   \R[B_n] \underset{\R[B_{n-i}\times B_i]}{\otimes} (1 \boxtimes \nu_i)$$
    has a basis corresponding to the cosets of $B_{n-i}\times B_i$ in $B_n$. As $B_{n-i}\times B_i=(N_{n-i}S_{n-i})\times (N_iS_i)=N_n(S_{n-i}\times S_i)$,    we have an isomorphism of sets
    $$ B_n/(B_{n-i}\times B_i) = (N_nS_n)/(N_n (S_{n-i}\times S_i)) \cong S_n/(S_{n-i} \times S_i)$$
    
    These cosets correspond to the subsets of $\{1,\dots,n\}$ with cardinality $i$. For each such subset $A$, let us choose a permutation $\vartheta_A\in S_n$ such that $\vartheta_A(\{n-i+1,\dt,n\})=A$. Then these permutations are representatives of the cosets $B_n/(B_{n-i}\times B_i)$, and therefore form a basis to the induced representation. 

    For any $B \subseteq \{1,\dt n\}$ and $A\subset \{1,\dt,n\}$ such that $|A|=i$, we have
    \begin{eqnarray*}
        s_B \cdot \vartheta_A &=& \vartheta_A (\vartheta_A^{-1} s_B \vartheta_A) = \vartheta_A s_{\vartheta_A^{-1}(B)}\\
        &=& \vartheta_A \cdot (1\boxtimes \nu_i)( s_{\vartheta_A^{-1}(B)} ) \\
        &=& \vartheta_A \cdot (-1)^{|\vartheta_A^{-1}(B) \cap \{n-i+1,\dt, n\}|} \\
        &=& \vartheta_A \cdot (-1)^{|B \cap \vartheta_A(\{i+1,\dt,n\})|} 
        =\vartheta_A \cdot (-1)^{|A\cap B|}
    \end{eqnarray*}

    For any $\pi \in S_n$ and $A\subset \{1,\dt,n\}$ such that $|A|=i$, we have 
    $$ (\pi \vartheta_A) (\{n-i+1,\dt,n\})=\pi(A)= \vartheta_{\pi(A)}(\{n-i+1,\dt,n\})$$  
    Hence $\pi' :={\vartheta_{\pi(A)}}^{-1} \pi \vartheta_A  \in S_{n-i} \times S_i$, so in the induced representation $V_{([n-i],[i])}$ we have
    $$ \pi \vartheta_A = \vartheta_{\pi(A)} \pi' = \vartheta_{\pi(A)} (1\boxtimes \nu_i)(\pi')=\vartheta_{\pi(A)}$$
 
    Therefore, the action of the permutations and the diagonal elements $s_A$ on the coset representatives $\vartheta_A$ is compatible with their action on the sets of cardinality $i$ in definition \ref{def:Vni}. This completes the proof.
\end{proof}

\section{Proof of Theorems \ref{thm:main} and \ref{thm:stronger_main}} \label{sec:proof of main}
Our proof of Theorem \ref{thm:stronger_main} relies on the following lemma:
\begin{lemma}\label{lem:no_trivial}
    Let $\sigma$ be an irrep of $B_n$. Then the restriction $\Res^{B_n}_{S_n} \sigma$ contains a copy of the trivial irrep of $S_n$ if and only if $\sigma$ corresponds to the pair $([n-i],[i])$ for some $0\leq i \leq n$.
\end{lemma}
\begin{proof}
    By Frobenius reciprocity, the $B_n$--irreps whose restriction to $S_n$ contains the trivial $S_n$--irrep are exactly those which appear in the decomposition of $\Ind_{S_n}^{B_n}1_{S_n}$ to $B_n$--irreps. The latter representation is the coset representation $\R[B_n/S_n]$. As $B_n= N_n S_n$, the elements of $N_n$ form a complete set of representatives of $B_n/S_n$. We shall prove that the irreducible components of the coset representation are exactly the irreps $([n-i],[i])$ for $0\leq i \leq n$ by constructing an explicit isomorphism

$$ \phi: \R[B_n / S_n] \xrightarrow{\sim} \bigoplus_{0\leq i \leq n} V_n^i
$$
    Our map will be defined by 
    \begin{equation*}
        \phi(s_C S_n)=\sum_{D\subseteq \{1,\dots,n\}}(-1)^{|C\cap D|} D \text{ for all }C\subseteq\{1,\dots,n\}.
    \end{equation*}

    As a map between vector spaces, $\phi$ is nothing but the Fourier transform in $N_n$, therefore it is an isomorphism. Let us verify that this is also a map of $B_n$--representations. We shall denote by $\oplus$ the symmetric difference operation between sets. 
    For any $A\subseteq\{1,\dots,n\}$, we have
    \begin{eqnarray*}
        \phi(s_A \cdot (s_C S_n)) &=& \phi(s_{A \oplus C}S_n) = \sum_{D\subseteq \{1,\dots,n\}}(-1)^{|(A\oplus C)\cap D|}s_{D} \\ 
        &=& \sum_{D\subseteq \{1,\dots,n\}}(-1)^{|A\cap D|}(-1)^{|C\cap D|}s_{D} \\ 
        &=& s_A \cdot ( \sum_{D\subseteq \{1,\dots,n\}}(-1)^{|C\cap D|}s_{D} ) \\
        &=& s_A \phi (s_C S_n)
    \end{eqnarray*}
    Also, for $\pi\in S_n$, we get
    \begin{eqnarray*}
        \phi(\pi \cdot (s_C S_n)) &=& \phi(\pi s_C \pi^{-1} \cdot (\pi S_n)) =\phi( s_{\pi(C)}S_n)\\ 
        &=&\sum_{D\subseteq \{1,\dots,n\}}(-1)^{|\pi(C)\cap D|} D\\
        &=&\sum_{D\subseteq \{1,\dots,n\}}(-1)^{|C\cap \pi^{-1}(D)|} D\\
        &=&\sum_{D\subseteq \{1,\dots,n\}}(-1)^{|C\cap D|} \pi(D)\\
         &=&\pi\cdot\phi(s_C S_n). 
    \end{eqnarray*}
\end{proof}

We now recall and prove Theorems \ref{thm:stronger_main} and \ref{thm:main}:

\begin{customthm}{\ref{thm:stronger_main}}
    Let $\sigma$ be a nontrivial irrep of $B_n$ which does \textbf{not} correspond to any element of $\mathcal F_n$, and let $w=w_T+w_N$, where $w_T$ and $w_N$ be as in Theorem \ref{thm:main}. Then we have
    $$ \psi_{B_n}(w,\sigma) \geq \psi_{B_n}(w, ([n-1,1],\emptyset)). $$
\end{customthm}
\begin{proof}
    As $\sigma$ is nontrivial and is not in $\mathcal F_n$, it is not of the form $([n-i],[i])$ for any $0\leq i \leq n$. By Lemma \ref{lem:no_trivial}, the restriction of $\sigma$ to $S_n$ contains no copy of the trivial irrep. Therefore there exist nontrivial $S_n$--irreps $\mu_1,\dt,\mu_k$ such that $\Res^{B_n}_{S_n} \sigma \cong \oplus_{i=1}^k \mu_i$.
    Hence:

    \begin{eqnarray*}
\psi_{B_n}(w, \sigma) & \stackrel{\left(1\right)}{\ge} & \psi_{B_n}(w_T, \sigma) \\
                      & \stackrel{\left(2\right)}{=} & \psi_{S_n}(w_T, \Res^{B_n}_{S_n} \sigma) \\
                      & \stackrel{\left(3\right)}{=} & \min_{1\leq i\leq k}\psi_{S_n}(w_T, \mu_i) \\
                      & \stackrel{\left(4\right)}{\ge} & \psi_{S_n}(w_T, [n-1,1]) \\
                      & \stackrel{\left(5\right)}{=} & \psi_{B_n}(w_T, ([n-1,1],\emptyset) ) \\
                      & \stackrel{\left(6\right)}{=} & \psi_{B_n}(w, ([n-1,1],\emptyset) )
\end{eqnarray*}
where (1) follows since $\sigma(\Delta(w_N))$ is a nonnegative operator; (2) follows since $w_T$ is supported on $S_n$, and $\Res^{B_n}_{S_n}\sigma$ does not contain copies of the trivial $S_n$--irrep; (3) is immediate; (4) follows from the Aldous-Caputo-Liggett-Richthammer theorem; (5) is the same as (2) where $\sigma=([n-1,1],\emptyset)$ (and also using Lemma \ref{lem:(sigma,empty)}); and (6) holds as the action matrix of $\Delta(w_N)$ on the representation $([n-1,1],\emptyset)$ is equal to $0$.
\end{proof}

\begin{customthm}{\ref{thm:main}}
    Let $w_T=\sum_{1\leq i<j\leq n} a_{ij}(ij)$ and $w_N=\sum_{A \subseteq \{1,\dt,n\}}\alpha_As_A$ with any nonnegative weights $(a_{ij})$ and $(\alpha_A)$. Let $w_N^-=\sum_{A \subseteq \{1, \dt, n \} \, : \, |A| \text { is odd} } \alpha_A s_A$.
    
    Then we have:
    \begin{enumerate} 
    \item $ \psi_{B_n}(w_T+w_N)=\min \{ \psi_{B_n}(w_T+w_N, \sigma) \st \sigma \in \mathcal F_n\}.$
    \item $ \psi_{B_n}(w_T + w_N^-)=\min \{ \psi_{B_n}(w_T+w_N^-, \sigma) \st \sigma \in \fnm\}.$
    \end{enumerate}
\end{customthm}
    
\begin{proof}
    The first part follows directly from Theorem \ref{thm:stronger_main}. For the second part, let $w=w_T+w_N^-$. Let us assume, without loss of generality, that $\alpha_A=0$ for any $A$ of even size, so $w_N=w_N^-$.
    
    By the first part, we have $ \psi_{B_n}(w)=\min \{ \psi_{B_n}(w, \sigma) \st \sigma \in \mathcal F_n\}.$ Let $\sigma_0=(\emptyset,[n])$. Since $ \fnm = \mathcal F_n \setminus \{\sigma_0\}$, it suffices to show that there exists $\tau \in \fnm$ such that $\psi_{B_n}(w,\tau) \leq \psi_{B_n}(w, \sigma_0)$. In fact, we shall now prove that this holds for any $\tau=([n-k],[k])$ with $1 \leq k \leq n-1$.

    By \eqref{eq:Bn_irrep_def}, $\sigma_0$ is the one-dimensional representation given by $\nu_n$ (see also \eqref{eq:nu_k_rep}). Since $\nu_n(s_A)=-1$ for any $A$ of odd size, and $\nu_n(\pi)=1$ for any $\pi \in S_n$, we have
    \begin{equation} \label{eq:psi of sigma0}
     \psi_{B_n}(w,\sigma_0)= \nu_n(\Delta(w))=\sum_{A\subseteq \{1,\dt, n\}}\alpha_A(1-\nu_n(s_A))=2\sum_{A\subseteq \{1,\dt, n\}} \alpha_A
    \end{equation}    
    By Lemma $\ref{lem:[n-i],[i]}$, $\tau$ is isomorphic to the irrep $V_n^k$. Consider the standard inner product on $V_n^k$:

    \begin{equation}
    \label{eq:inner product}
    \langle \sum \alpha_S S, \sum \beta_S S \rangle = \sum \alpha_s \beta_s
    \end{equation}

    As $w$ is symmetric, so is $\Delta(w)$, and therefore the matrix representing its action on the basis $\{S \subseteq B \st |S|=k\}$ is symmetric. Therefore, the minimum eigenvalue of the action of $\Delta(w)$ on $V_n^k$ is given by the minimum Rayleigh quotient:
    \begin{equation}
    \label{eq:rayleigh}
    \psi_{B_n}(w,V_n^k)=\min_{\xi \in V_n^k \setminus \{0\}} \frac{\langle \Delta(w) \xi, \xi \rangle }{\langle \xi, \xi \rangle}
    \end{equation}
 Let us consider the element
    $$ v=\sum_{S \subseteq \{1,\dt,n\}, |S|=k}S \in V_n^k$$
As $v$ is $S_n$-invariant, we have $\Delta(w_T)v=0$. For any $A \subseteq \{1,\dt,n\}$ we have by \eqref{eq:sA_action}, $$s_Av=\sum_{S \subseteq \{1,\dt,n\}, |S|=k}(-1)^{|A\cap S|}S$$
and therefore,
\begin{equation*}
 \langle s_Av,v \rangle= \sum_{S \subseteq  \{1,\dt,n\}, |S|=k}(-1)^{|A\cap S|}\geq \sum_{S \subseteq  \{1,\dt,n\}, |S|=k}(-1)=-\binom nk.
 \end{equation*}
We also have 
\begin{equation} \label{eq:v,v}
    \langle v,v\rangle = \binom nk
\end{equation}
Therefore:
\begin{eqnarray*}
\langle \Delta(w)v, v\rangle &=& 
\langle \Delta(w_T)v,v\rangle + \langle \Delta(w_N^-)v,v\rangle \\
&=& \langle \Delta(w_N^-)v,v\rangle 
= \sum_{A\subseteq \{1,\dt, n\}} \alpha_A \langle v-s_A v,v \rangle \\
&=&  \sum_{A\subseteq \{1,\dt, n\}} \alpha_A (\langle v,v \rangle -\langle s_A v,v \rangle) \leq 2 \binom nk \sum_A \alpha_A
\end{eqnarray*}
By \eqref{eq:psi of sigma0}, \eqref{eq:rayleigh} and \eqref{eq:v,v}, we conclude:
$$ \psi_{B_n}(w,([n-k],[k])) \leq
\frac{\langle \Delta(w)v,v \rangle}{\langle v,v \rangle}
\leq \frac{2 \binom nk \sum_A \alpha_A}{\binom nk} = \psi_{B_n}(w, \sigma_0).
$$

\end{proof}

\section{Proof of Theorem \ref{thm:required_irreps}} \label{sec:required irreps}

We now turn to showing by explicit construction that for any $n$, all the irreps in $\mathcal F_n$ (resp. $\fnm$) are needed in order to account for the spectral gap of the general group ring elements of the form $w_T+w_N$ (resp. $w_T+w_N^-$). 

We start with another lemma on the branching of $B_n$--irreps.

\begin{lemma} \label{lem:Nn-Branching}
    Let $V=V_{(\sigma_1,\sigma_2)}$ be a $B_n$--irrep. Then $\Res^{B_n}_{N_n}V$ contains a copy of the trivial $N_n$--irrep if and only if $\sigma_2=\emptyset$.
\end{lemma}

\begin{proof}
    We use the same reasoning as in the proof of Lemma \ref{lem:no_trivial}. By Frobenius reciprocity, the $B_n$--irreps whose restriction to $N_n$ contains the trivial $N_n$--irrep are exactly those which appear in the decomposition of $\Ind_{N_n}^{B_n}1_{N_n}$ to $B_n$--irreps. The latter representation is the coset representation $\R[B_n/N_n]$. Since $B_n/N_n \cong S_n$, this representation is exactly the pullback of the regular representation of $S_n$ to $B_n$. By Lemma \ref{lem:(sigma,empty)}, its irreducible components are exactly the representations $V_{(\sigma_1,\sigma_2)}$ for which $\sigma_2 = \emptyset$.
\end{proof}

We shall also need the following lemma:

\begin{lemma} \label{lem:complete_graph_ev}
    Let $w=\sum_{1\leq i < j \leq n}(ij)$. Then $\psi_{S_n}(w,[n-1,1])=n$.
\end{lemma}

\begin{proof}
    As $w$ is a central element in $S_n$, it acts as a scalar on $[n-1,1]$, and so does $\Delta(w)$. The scalar is equal to the trace of $\Delta(W)$ on $[n-1,1]$, divided by the dimension of $[n-1,1]$. The trace of $\Delta(w)$ on $[n-1,1]$ is the same as its trace on the standard representation $[n-1,1] \oplus [1]$. A direct matrix calculation shows that the trace of $\Delta((ij))$ on the standard representation is $2$. We therefore have
    $$ \psi_{S_n}(w,[n-1,1]) = \frac 1{n-1}\binom n2 \cdot 2= n$$
\end{proof}

We now recall and prove:

\begin{customthm}{\ref{thm:required_irreps}}
\begin{enumerate}
\item
    Let $\sigma \in \mathcal F_n$. Then there exist $w_T, w_N$ as in Theorem \ref{thm:main} such that $\tau=\sigma$ is the unique minimizer of the expression $\psi_{B_n}(w_T+w_N,\tau)$ among all nontrivial irreps $\tau$ of $B_n$.
\item 
    Let $\sigma \in \fnm$. Then there exist $w_T, w_N^-$ as in Theorem \ref{thm:main} such that $\tau=\sigma$ is the unique minimizer of the expression $\psi_{B_n}(w_T+w_N^-,\tau)$ among all nontrivial irreps $\tau$ of $B_n$.    
    \end{enumerate}
\end{customthm}

\begin{proof}
Let us start with the second part of the theorem, and with the main case $\sigma=([n-k],[k])$, where $1\leq k \leq n-1$. 
Let $A=\{1,2,\dt,k\}$ and $B=\{1,2,\dt n\}$. Consider the collection
    $$ U = \{ S \subseteq B \st |S| \text{ is odd and } |S \cap A| \text{ is even}\}.$$
    Let $w_N^-=4\sum_{C\in U} s_C$ and $w_T=\frac {2^n}n\sum_{1\leq i < j \leq n}(ij)$. We will show that $w=w_T+w_N^-$ satisfies the desired property, namely, that $\psi_{B_n}(w,\sigma) < \psi_{B_n}(w, \tau)$ for any nontrivial $B_n$--irrep $\tau \ncong \sigma$.

    For any set $C\subseteq B$, let $\chi_C: \mathcal P(B)\rightarrow \mathbb R$ be the character $$\chi_C(D)=(-1)^{|C \cap D|},$$ and let $\rho_C$ be the group ring element $$ \rho_C = \sum_{D \subseteq B} \chi_C(D) s_D \in \R[B_n]$$
    Note that the characteristic function of $U$ is
    \begin{eqnarray*}
     \delta_U &=& \frac 14 (\chi_\emptyset + \chi_A)(\chi_\emptyset - \chi_B) \\
     &=& \frac 14 (\chi_\emptyset + \chi_A - \chi_B - \chi_{B \setminus A})
     \end{eqnarray*}
    We therefore have
    \begin{equation} \label{eq:wN_by_characters}
    w_N^-=4\sum_{D\subseteq B} \delta_U(D)s_D= \rho_\emptyset +\rho_A-\rho_B-\rho_{B\setminus A}
    \end{equation}
    
    Since $|U|=2^{n-2}$, we have
    \begin{equation} \Delta(w_N^-)= 2^n s_\emptyset - w_N^- \label{eq:Delta_wN} \end{equation}
    Let us define, for any element $e=\sum_{D\subseteq B} a_D s_D \in \R[N_n]$ its Fourier transform $\hat{e}$ as the function from $\mathcal P(B)$ to $\R$ given by
    $$ \hat e(S)=\sum_{D\subseteq B} (-1)^{|S\cap D|}a_D $$
    By (\ref{eq:sA_action}), the action of $e$ on $V_n^i$ is diagonal with respect to the basis given by subsets of $B$ of size $i$, and the eigenvalue corresponding to the set $S$ is $\hat e(S)$:
    \begin{equation} \label{eq:action_by_fourier}
        e\cdot S=\hat{e}(S)S
    \end{equation}

    The following Fourier transforms are straightforward:
    \begin{eqnarray}
        \widehat{\rho_C}(S)=\begin{cases}
2^n & \text{if } S=C, \\
0 & \text{otherwise} \end{cases} \label{eq:rho_hat}\\ 
    \widehat{s_\emptyset}(S)=1 \text{ (for all $S$)} \label{eq:empy_hat}
    \end{eqnarray}

    By (\ref{eq:wN_by_characters})-(\ref{eq:empy_hat}), we conclude that the Fourier transform of $\Delta (w_N^-)$ is

    \begin{equation}\label{eq:Fourier_of_Delta_w} \widehat{\Delta(w_N^-)}(S)= \begin{cases}
0 & \text{if } S=A \text{ or } S=\emptyset, \\
2\cdot 2^n& \text{if } S=B \text{ or } S=B \setminus A, \\
2^n & \text{otherwise}
\end{cases}
    \end{equation}
     In particular, we get for all $1\leq i \leq n$ with $i\neq k$:
    \begin{equation*}
        \psi_{B_n} (w_N^-, ([n-i],[i]))=\min_{|S|=i}\widehat{\Delta(w_N^-)}(S) \geq 2^n
    \end{equation*}
    Since $\Delta(w_T)$ is a nonnegative operator, we conclude that for such $i$:
    \begin{equation} \label{eq:lower_bound_other_irreps}
        \psi_{B_n} (w, ([n-i],[i])) \geq 2^n
    \end{equation}    
    Next, we will find an upper bound for $\psi_{B_n}(w,([n-k],[k]))$. Let us consider (as in the proof of Theorem \ref{thm:main}) the element
    $$ v=\sum_{S \subseteq B, |S|=k}S \in V_n^k$$

    For any $S\subseteq B$ of size $k$ and any $C \subseteq B$ we have by (\ref{eq:action_by_fourier}) and (\ref{eq:rho_hat}):
    $$ \rho_C\cdot S = \widehat{\rho_C}(S)S = 
    \begin{cases} 2^n S & \text{if } S=C \\ 0 & \text{otherwise} \end{cases}$$
    Therefore,
    \begin{equation}\label{eq:rhoCv} \rho_C \cdot v=\begin{cases} 2^n C & \text{if } |C|=k \\ 0 & \text{otherwise} \end{cases}\end{equation}
    By (\ref{eq:wN_by_characters}),(\ref{eq:Delta_wN}) and (\ref{eq:rhoCv}),
    \begin{equation} \label{eq:Delta_wN_v}
     \Delta(w_N^-) \cdot v = \begin{cases} 2^nv +2^n {(B \setminus A)} -2^n A & \text{if } 2k=n \\
    2^nv-2^nA & \text{otherwise}\end{cases}
    \end{equation}

    As $v$ is $S_n$-invariant, we have $w_T v=v$, hence 
    \begin{equation}\label{eq:Delta_v_0}\Delta(w_T)v=0 \end{equation}
    
    Recall the standard inner product on $V_n^k$ (\eqref{eq:inner product}). By (\ref{eq:Delta_wN_v}) and (\ref{eq:Delta_v_0}),

    \begin{eqnarray*} \langle \Delta(w) v, v \rangle
    &=& \langle (\Delta(w_T) + \Delta(w_N)) v,v \rangle \\
    &=& \langle \Delta(w_N) v  ,v \rangle \\
    &=& \begin{cases} 2^n\langle v + {(B \setminus A)} - A,v \rangle & \text{if } 2k=n \\
    2^n \langle v-A ,v \rangle & \text{otherwise}\end{cases} \\
    &=& \begin{cases} 2^n\binom nk  & \text{if } 2k=n \\
    2^n (\binom nk -1)& \text{otherwise}\end{cases}
    \end{eqnarray*}

    Hence, $\langle \Delta(w) v, v \rangle \leq 2^n \binom nk$. 
    As $\langle v,v\rangle =\binom nk$, we get by the Rayleigh quotient principle (\eqref{eq:rayleigh}):
    \begin{equation} \label{lambda_estimate}
        \psi_{B_n}(w, V_n^k) \leq \frac{\langle \Delta(w) v, v \rangle}{\langle v, v \rangle}\leq 2^n
    \end{equation}
    In fact, we have a strong inequality $\psi_{B_n}(w, V_n^k) < 2^n$, as by (\ref{eq:Delta_wN_v}), $\Delta(w) v = \Delta(w_N^-) v$ is not a multiple of $v$, so $v$ is not an eigenvector of $\Delta(w)$.

    Together with (\ref{eq:lower_bound_other_irreps}), we get that $\psi_{B_n}(w, ([n-k],[k]))$ is the unique minimum among the numbers $\psi_{B_n}(w, ([n-i],[i]))$, for $1\leq i \leq n$.

    Finally, by Lemma \ref{lem:complete_graph_ev}, 
    $$ \psi_{B_n}(w,([n-1,1],\emptyset))\geq \psi_{B_n}(w_T,([n-1,1],\emptyset))= \frac{2^n}n \cdot n = 2^n$$ so
    \begin{equation}\label{eq:std_win}
    \psi_{B_n}(w,([n-k],[k])) < \psi_{B_n}(w,([n-1,1],\emptyset))
    \end{equation}

    We conclude that $\psi_{B_n}(w,\sigma)<\psi_{B_n}(w,\tau)$ for all $\tau\in \mathcal F_n \setminus \{\sigma\}$.
    For any nontrivial $\tau \notin \mathcal F_n$, we have by (\ref{eq:std_win}) and Theorem \ref{thm:stronger_main}:
    $$ \psi_{B_n}(w,\tau) \geq \psi_{B_n}(w,([n-1,1],\emptyset)) > \psi_{B_n}(w, \sigma)$$

    This concludes the case $\sigma=([n-k],[k])$ in the second part of the theorem. 
    
    Let us now assume that $\sigma = ([n-1,1],\emptyset)$. 
    In this case we set $w_N^-=n\sum_{i=1}^n s_{\{i\}}$ and $w_T=\sum_{1\leq i < j \leq n}(ij)$.
    Let us find $\psi(w_N^-,([n-i],[i]))$ for each $i$. By \eqref{eq:sA_action}, if $A\subseteq \{1,\dt,n\}$ is a set of size $i$, viewed as as an element of $V_n^i$, then
$$ \Delta(w_N^-) A = n\sum_{i=1}^n (1- (-1)^{|\{i\}\cap A|})A=2n|A|\cdot A $$
Therefore, all the eigenvalues are equal to $2ni$, so
\begin{equation} \label{eq:psi>=2n}
\psi_{B_n}(w_N^-, V_n^i)=2ni\geq 2n.
\end{equation}

Let us now consider $w=w_N^-+w_T$. The action of $w_N^-$ on $\sigma$ is equal to the zero operator. Hence, by Lemmas \ref{lem:(sigma,empty)} and \ref{lem:complete_graph_ev},
$$ \psi_{B_n}(w, \sigma) = \psi_{B_n}(w_T, \sigma) = n.$$
Let $\tau \neq \sigma$ be any nontrivial irrep of $B_n$. $\tau$ corresponds to a pair of partitions, $(\tau_1,\tau_2)$. We consider two cases:
\begin{itemize}[left = 0pt]
    \item \textbf{Case 1: $\tau_2 \neq \emptyset$}. In that case, by Lemma \ref{lem:Nn-Branching}, 
    $\Res^{B_n}_{N_n} \tau$ does not contain a copy of the trivial $N_n$--irrep. As $N_n$ is abelian, $\Res^{B_n}_{N_n} \tau$ decomposes into nontrivial one-dimensional irreps. By \eqref{eq:sA_action}, these are all included in $\bigoplus_{i=1}^n V_n^i$. Therefore, using \eqref{eq:psi>=2n}, we have
    \begin{eqnarray*}
    \psi_{B_n}(w, \tau) &\geq& 
     \psi_{B_n}(w_N^-,\tau) 
     = \psi_{N_n}(w_N^-,\Res^{B_n}_{N_n} \tau) \\
     &\geq& \min_{1\leq i \leq n} \psi_{N_n}(w_N^-,\Res^{B_n}_{N_n}V_n^i)\\
    &=& \min_{1\leq i \leq n} \psi_{B_n}(w_N^-,V_n^i) 
    \geq  2n > \psi_{B_n}(w,\sigma).
    \end{eqnarray*}

    \item \textbf{Case 2: $\tau_2 = \emptyset$}. In that case, by Lemma \ref{lem:(sigma,empty)}, 
    \begin{eqnarray*}
    \psi_{B_n}(w,\tau)&=&\psi_{B_n}(w_T,\tau)=\psi_{S_n}(w_T, \Res ^{B_n}_{S_n} \tau)\\
    &=& \psi_{S_n}(w_T, \tau_1)
    \end{eqnarray*}
    Similarly, we have
    $$ \psi_{B_n}(w,\sigma) =\psi_{S_n}(w_T, [n-1,1])$$
    Therefore, we need to prove that
    \begin{equation} \label{eq:tau_1>[n-1,1]} \psi_{S_n}(w_T, \tau_1)>\psi_{S_n}(w_T, [n-1,1]) \end{equation} Note that since $\tau \ncong \sigma$ and $\tau$ is nontrivial, $\tau_1 \neq [n]$ and $\tau_1 \neq [n-1,1]$. 

    For any Young diagram $n \vdash \lambda$, let $V_\lambda$ be a corresponding $S_n$--irrep. Let $r(\lambda)$ be the character ratio $r(\lambda)=\chi_{\lambda}((12))/d_\lambda$, where $d_\lambda= \dim V_\lambda$. Since $w_T$ is in the center of $S_n$, all the eigenvalues of its action on $V_\lambda$ are equal to $\binom n2 r(\lambda)$. Therefore \eqref{eq:tau_1>[n-1,1]} is equivalent to 
    
    \begin{equation}\label{eq:r(tau1)<...} r(\tau_1) < r([n-1,1]) \end{equation}
    
    Let $\succeq$ be the dominance partial order on Young diagrams (a definition can be found, e.g., in \cite{DS} before Lemma 9). By Lemma 10 in \cite{DS}, if $n\vdash \lambda, \delta$ and $ \lambda \succeq \delta$ then $r(\lambda)\geq r(\delta)$. The proof of that lemma shows that in the case where $\lambda \neq \delta$ we have a strict inequality $r(\lambda)> r(\delta)$. 

    We have $[n-1,1]\succ \tau_1$, and therefore, (\ref{eq:r(tau1)<...}) holds. This concludes the proof of the second part of the theorem.
\end{itemize}
    As for the first part of the theorem, let $\sigma \in \mathcal F_n$. If $\sigma \in \fnm$ then we are done by the second part of the theorem. Therefore, it is enough to prove the claim for $\sigma=\sigma_0=(\emptyset,[n])$. For that matter, let us define $w=w_T+w_N$ where $w_T=\sum_{1\leq i < j \leq n}(ij)$ and $w_N=\sum_{B \subseteq \{1,\dt,n\},\, |B|=2}s_B$. As we have seen in the proof of Theorem \ref{thm:main},
    $$ \psi_{B_n}(w, \sigma_0)=\nu_n(w)=0.$$

Therefore, it suffices to show that $\psi_{B_n}(w, \tau)>0$ for any nontrivial $B_n$-irrep $\tau \ncong \sigma_0$. Let us consider the following cases:

\begin{itemize}
\item
\textbf{Case 1:} $\tau=([n-k],[k])$ for some $1\leq k \leq n-1$. For each $A \subseteq \{1,\dt,n\}$ of size $k$, and $B\subseteq \{1,\dt,n\} $ of size $2$, we have
$$ \Delta(s_B)A= (1-(-1)^{|A\cap B|})A=\begin{cases} 
2A & \text{if } |A\cap B|=1 \\
0 & \text{otherwise} \end{cases} $$
For each such $A$ there are $k(n-k)$ subsets $B$ such that $|A \cap B|=1$. Therefore,
$$ \Delta(w_N)A=\sum_{B \subseteq \{1,\dt,n\},\, |B|=2} \Delta(s_B)A=2k(n-k)A.$$
We conclude that $\Delta(w_N)$ is a scalar operator on $V_n^k$ with a single eigenvalue $2k(n-k)$. Therefore,
$$ \psi(w,\tau) \geq \psi(w_N,\tau)=2k(n-k) >0.$$
\item
\textbf{Case 2:} $\tau=([n-1,1],\emptyset)$. By Lemmas \ref{lem:(sigma,empty)} and \ref{lem:complete_graph_ev},
$$ \psi_{B_n}(w,\tau)\geq \psi_{B_n}(w_T,\tau)=
\psi_{S_n}(w_T,[n-1,1])=n>0.$$
 
\item
\textbf{Case 3:} $\tau \notin \fnm$. We have assumed that $\tau \ncong \sigma_0$, therefore $\tau \notin \mathcal F_n$. By Theorem \ref{thm:stronger_main} and Case 2:
$$ \psi_{B_n}(w,\tau) \geq \psi_{B_n}(w, ([n-1,1],\emptyset)) > 0.$$
\end{itemize}
As all the cases were covered, our proof is complete.
\end{proof} 
\section{Concluding remarks} \label{sec:conclusion}

In this paper we have shown that for general elements of the form $w=w_T+w_N$, the spectral gap $\psi_{B_n}(w)$ always comes from one of the $n+1$ irreps in the family $\mathcal F_n$, and that all of these irreps are needed in order to cover for all elements $w$ of this type.

Cesi's result, on the other hand, shows that when $w_N$ is supported on sets of size $1$, the spectral gap comes from one of two representations, of total dimension $2n-1$. This leaves room for questions for future study. Let us consider an element $w=w_T+w_N \in \R[B_n]$ where $w_T=\sum_{1\leq i < j \leq n} a_{ij}(ij)$ and $w_N=\sum_{A\subseteq \{1,\dt,n\}} \alpha_A s_A$. Let us say that $w$ is $k$--upper bounded if $w_N$ is supported on sets of size $\leq k$, and that $w$ is $k$--lower bounded if $w_N$ is supported on sets of size $\geq k$.

\begin{question} For any $1\leq k \leq n$, find a minimal collection of irreps such that the spectral gap for any $k$--upper bounded element comes from that collection. In particular, for any fixed $k$, is the size of such a minimal collection bounded as $n\rightarrow \infty$? Is the total dimension of such a minimal collection bounded by a polynomial in $n$?
\end{question}

Similarly, we can ask these questions about $k$--lower bounded elements:

\begin{question} For any $1\leq k \leq n$, find a minimal collection of irreps such that the spectral gap for any $k$--lower bounded element comes from that collection. In particular, for any fixed $k$, is the size of such a minimal collection for $(n-k)$--lower bounded elements bounded as $n\rightarrow \infty$? Is the total dimension of such a minimal collection for $(n-k)$--lower bounded elements bounded by a polynomial in $n$?
\end{question}

Answering these questions would probably require more refined techniques for construction of elements $w$ which require specific irreps, or new methods for comparing minimal eigenvalues between different irreps.

\subsection*{Acknowledgements} The second author was supported by a research grant ANRF/ECRG/
2024/006021/PMS.


\end{document}